\documentclass[12pt, reqno]{amsart}
\usepackage{amsmath, amstext, amsbsy, amssymb, amscd,mathrsfs}
\usepackage{amsmath}
\usepackage{amsxtra}
\usepackage{amscd}
\usepackage{amsthm}
\usepackage{amsfonts}
\usepackage{amssymb}
\usepackage{eucal}
\usepackage{color}

\input xy
\xyoption{all}

\input prepictex
\input postpictex

\newtheorem{theorem}{Theorem}[section]
\newtheorem*{thm}{Theorem}
\newtheorem{lemma}[theorem]{Lemma}
\newtheorem{proposition}[theorem]{Proposition}
\newtheorem{corollary}[theorem]{Corollary}
\theoremstyle{definition}

\newtheorem{remark}[theorem]{Remark}

\theoremstyle{comment}

\numberwithin{equation}{section}

\newcommand{\g}{\mathfrak{g}}
\newcommand{\n}{\mathfrak{n}}
\newcommand{\h}{\mathfrak{h}}

\newcommand{\rea}[2]{U_{#1}(#2)} 

\newcommand{\pth}[1]{{#1}^{[p]}} 

\newcommand{\ra}{\rightarrow}

\newcommand{\ev}[1]{{#1}_{\bar{0}}}
\newcommand{\od}[1]{{#1}_{\bar{1}}}
\newcommand{\osp}{\mathfrak{osp}}

\newcommand{\la}{\lambda}

\newcommand{\gl}{{\mathfrak{gl}}}
\newcommand{\Z}{ \mathbb Z }

\newcommand{\nc}{\newcommand}
 \nc{\La}{\Lambda}
 \nc{\ep}{\epsilon}
 \nc{\vep}{\varepsilon}

{\vskip-\lastskip\medskip
  \noindent
  {\em #1.}\enspace
  }%
{\qed\par\medskip
  }

\makeatletter
\def\Ddots{\mathinner{\mkern1mu\raise\p@
\vbox{\kern7\p@\hbox{.}}\mkern2mu
\raise4\p@\hbox{.}\mkern2mu\raise7\p@\hbox{.}\mkern1mu}}
\makeatother

\begin{document}
\title[Typical Blocks of Lie superalgebras]
{Typical Blocks of Lie Superalgebras in Prime Characteristic}

\author[Lei Zhao]{Lei Zhao}
\address{Department of Mathematics, University of Virginia,
Charlottesville, VA 22904} \email{lz4u@virginia.edu}

\begin{abstract}
For a type I basic classical Lie superalgebra $\g=\ev \g \oplus
\od\g$, we establish an equivalence between typical blocks of
categories of $\rea{\chi}{\g}$-modules and $\rea{\chi}{\ev
\g}$-modules. We then deduce various consequences from the
equivalence.
\end{abstract}

\maketitle
\date{}
  \setcounter{tocdepth}{1}

\section{Introduction}
The representation theory of Lie superalgebras over complex numbers
has been studied quite extensively in literature, starting from the
well-known character formula, due to Kac, for irreducible highest
weight representations with typical highest weights \cite{Kac2}. In
\cite{WZ1}, the author and Wang initiated the study of modular
representations of Lie superalgebras over an algebraically closed
field $K$ of characteristic $p>2$, by formulating a superalgebra
analogue of the Kac-Weisfeiler conjecture and establishing it for
the basic classical Lie superalgebras.

This paper is an attempt to study more closely and precisely the
structure of modular representations of Lie superalgebras; and it
deals with, as a natural starting point, the typical ones. In
particular, for a basic classical Lie superalgebra $\g=\ev \g \oplus
\od \g$ of type I (see \cite{Kac1} for a definition), we establish
an equivalence between the so-called typical blocks of $\g$-modules
over the field $K$ and the corresponding ones for $\ev \g$-modules.
Let us explain in more details.

Let $\g$ be a type I basic classical Lie superalgebra, i.e. $\g$ is
one of $\gl(m|n)$, $\mathfrak{sl}(m|n)$ or $\osp(2|2n)$. Then $\g$
is a restricted Lie superalgebra, and one can make sense the notions
of $p$-characters $\chi \in \ev \g^*$ and the corresponding reduced
enveloping superalgebras $\rea{\chi}{\g}$. Let $\g=\n^- \oplus \h
\oplus \n^+$ be a triangular decomposition of $\g$. We may assume
that a $p$-character $\chi$ satisfies $\chi(\ev \n^+)=0$ via the
conjugation of the adjoint group $\ev G$ of $\ev \g$ if necessary.

One of the distinguished features of $\g$ being type I is that $\g$
admits a natural $\Z$-grading
\[
\g=\g_{-1} \oplus \ev \g \oplus \g_1,
\]
with $\od \g =\g_{-1} \oplus \g_1$. Put $\g_+ = \ev \g \oplus \g_1$.
Let $M^0$ be a $\rea{\chi}{\ev \g}$-module. View it as a
$\rea{\chi}{\g_+}$-module with trivial $\g_1$-action. Then the
induction $\psi(M^0)= \rea{\chi}{\g}\otimes_{\rea{\chi}{\g_+}}M^0$
will give us a $\rea{\chi}{\g}$-module. The first main result
(Proposition~\ref{prop:Kac module}) of the paper is an
irreducibility criterion for $\psi(L^0)$ for an irreducible
$\rea{\chi}{\ev \g}$-module $L^0$, which states that the induced
module $\psi(L^0)$ is simple if and only the subspace $(L^0)^{\ev
\n^+}$ consists only weights which are typical (first defined in
\cite{Kac3}). It generalizes a classical result of Kac {\it loc.
cit.}.

Let $\rea{\chi}{\ev \g}$-$\text{mod}^{\text{ty}}$ (resp.
$\rea{\chi}{\g}$-$\text{mod}^{\text{ty}}$) be the full subcategory
of (finite-dimensional) $\rea{\chi}{\ev \g}$-modules (resp.
$\rea{\chi}{\g}$) whose composition factors are typical (see
Section~\ref{sec:ME} for the precise definition). The following is
the main theorem of the paper.

\begin{thm}[Theorem~\ref{thm:ME}]
The functor of taking $\g_1$-invariants,
\begin{equation*}
\qquad\phi=-^{\g_1}: \rea{\chi}{\g}\text{-}\text{mod}^{\text{ty}}
\rightarrow \rea{\chi}{\ev \g}\text{-}\text{mod}^{\text{ty}}
\end{equation*}
and the functor
\begin{equation*}\psi=\rea{\chi}{\g} \otimes_{\rea{\chi}{\g_+}} -:
\rea{\chi}{\ev \g}\text{-}\text{mod}^{\text{ty}} \rightarrow
\rea{\chi}{\g}\text{-}\text{mod}^{\text{ty}}
\end{equation*}
are inverse equivalence of categories.
\end{thm}

The theorem allows us to ``transfer'' various results on the
representation theory of $\ev \g$ to that of $\g$. First, we show
for a semisimple $p$-character $\chi$ a simplicity criterion of baby
Verma modules, and, as a consequence, a semisimplicity criterion for
$\rea{\chi}{\g}$; this generalizes a result of Rudakov \cite{Rud}.
Secondly, for nilpotent $p$-characters, we relate the
representations of regular typical blocks of $\rea{\chi}{\g}$ to the
Springer fiber of the flag manifold of $\ev G$, via a recent proof
of a conjecture of Lusztig by Bezrukavnikov, Mirkovi\'{c} and
Rumynin \cite{BMR}. Thirdly, for $\chi=0$, we deduce from Cline,
Parshall and Scott \cite{CPS} a character formula for irreducible
modules whose weights are typical and $p$-regular in terms of baby
Verma modules, assuming the original Lusztig conjecture \cite{Lus}.

Connections to \cite{BMR,CPS} indicate that the modular
representation theory of Lie superalgebras in typical blocks is
already very rich; the ``non-typical'' representations are expected
to exhibit new purely super phenomena.

The paper is organized as follows. In Section~\ref{sec:basics}, we
recall some basic facts on the modular representations of Lie
superalgebras. Section~\ref{sec:ME} is devoted to establishing the
equivalence of typical blocks. Finally in Section~\ref{sec:app}, we
deduce various consequences from the equivalence.

\noindent {\bf Acknowledgments.} The author is very grateful to his
advisor, Weiqiang Wang, for suggesting the problem as well as
offering valuable advice throughout the years.

\bigskip

\noindent{\bf Added note.} After the paper is submitted, the author
came across Chaowen Zhang's paper \cite{Zh} on arXiv, in which he
(independently) obtained the simplicity criterion for baby Verma
modules with semisimple $p$-characters for basic classical Lie
superalgebras.

\section{Preliminaries}\label{sec:basics}
\subsection{} Throughout we work with an algebraically
closed field $K$ with characteristic $p>2$ as the ground field.
Unless otherwise specified, $\g=\ev \g \oplus \od \g$ will be
denoted one of the Type I basic classical Lie superalgebras, i.e.,
$\gl(m|n)$, $\mathfrak{sl}(m|n)$ or $\osp(2|2n)$ over $K$. Let
$U(\g)$ denote the universal enveloping algebra of $\g$ over $K$.

For a finite-dimensional associative superalgebra $A$, let $A$-mod
denote the category of finite-dimensional $A$-modules. A
superalgebra analogue of Schur's Lemma states that the endomorphism
ring of an irreducible module in $A$-mod is either one-dimensional
or two-dimensional (in the latter case it is isomorphic to a
Clifford algebra). An irreducible module is said to be {\em of type
$M$} if its endomorphism ring is one-dimensional and it is said to
be {\em of type $Q$} otherwise.

By vector spaces, derivations, subalgebras, ideals, modules, and
submodules etc. we mean in the super sense unless otherwise
specified.

\subsection{} Let $\h$ be a Cartan subalgebra of $\g$ (and of $\ev \g$).  Denote
by $\Delta$ the set of roots of $\g$ relative to $\h$ and by
$\Delta_{\bar{0}}$ (resp. $\Delta_{\bar{1}}$) the set of even (resp.
odd) roots.

Fix a Borel subalgebra $\mathfrak{b}_{\bar{0}}$ of $\ev \g$
containing $\h$ and let $\mathfrak{b}=\mathfrak{b}_{\bar{0}} \oplus
\mathfrak{b}_{\bar{1}}$ be a distinguished Borel subalgebra $\g$ in
the sense of \cite[Section~1.4]{Kac3}. We have a triangular
decomposition of $\g$:
\[
\g =\n^- \oplus \h \oplus \n^+, \quad \mathfrak{b}=\h \oplus \n^+,
\]
where $\n^+$ (resp. $\n^-$) is the nilradical (resp. opposite
nilradical) of $\mathfrak{b}$. Denote the set of positive even
(resp. odd) roots by $\Delta_{\bar{0}}^+$ (resp.
$\Delta_{\bar{1}}^+$) corresponding to the triangular decomposition,
and let $\Pi$ be the set of simple roots. One can easily extend a
Chevalley basis of $[\ev \g,\ev \g]$ to an integral basis
$\{X_\alpha, h_i |\; \alpha \in \Delta, 1 \leq i \leq r\}$ of $\g$.

Let $W$ be the Weyl group of $\ev \g$. Denote by $(.|.)$ a
$W$-invariant bilinear form on $\h^*$. The dot action of $W$ on
$\h^*$ is defined by
\[
w.\la :=w(\la + \rho)-\rho,
\]
for $\la \in \h^*$, where $\rho=\frac{1}{2}(\sum_{\alpha \in
\Delta_{\bar{0}}^+} \alpha - \sum_{\beta \in \Delta_{\bar{1}}^+}
\beta)$. Note that for $w\in W$, $w(\rho_1)=\rho_1$ and so
\[
w.\la=w(\la+\rho_0)-\rho_0,
\]
where $\rho_0 =\frac{1}{2} \sum_{\alpha \in \Delta_{\bar{0}}^+}
\alpha$.

Being type I, the Lie superalgebra $\g$ admits a natural
$\Z$-grading compatible with the Lie superalgebra structure as
follows:
\[
\g =\g_{-1} \oplus \ev \g \oplus \g_1, \quad \text{with } \od
\g=\g_{-1}\oplus \g_1,
\]
where $\g_{1}$ (resp. $\g_{-1}$) is spanned by the root vectors
$X_\alpha$ with $\alpha \in \Delta_{\bar{1}}^+$ (resp.
$-\Delta_{\bar{1}}^+$). Finally, put
$$\g_+ =\ev \g \oplus \g_1.$$

\subsection{}\label{sec:elerepn} Our Lie superalgebra $\g$ is restricted in the sense
that its even subalgebra $\ev \g$ is a restricted Lie algebra with
$p$th power map $(-)\pth { } :\ev \g \ra \ev \g$; and the odd part
$\od \g$ is a restricted module by the adjoint action of $\ev \g$.
We refer the reader to \cite{WZ1} for the basic results of
representation theory of restricted Lie superalgebras.

For each $\chi \in \ev \g^*$, the {\em reduced enveloping
superalgebra} of $\g$ with the $p$-character $\chi$ is by definition
the quotient of $U(\g)$ by the ideal $I_{\chi}$ generated by all
$x^p -\pth x -\chi(x)^p$ with $x \in \ev \g$. It has a basis
consisting of products
\begin{equation}\label{equ:basis}
\prod_{\alpha \in \Delta_{\bar{0}}}X_\alpha^{a_{\alpha}} \cdot
\prod_i h_i^{b_i}\cdot \prod_{\beta \in
\Delta_{\bar{1}}}X_\beta^{c_\beta},
\end{equation}
with $0 \leq a_\alpha, b_i<p$ and $c_\beta =0,1$. Each
$\rea{\chi}{\g}$ has dimension equal to $p^{\dim \ev \g} 2^{\dim \od
\g}$.
%
%

Let $\ev G$ be the adjoint algebraic group of $\ev \g$. As usual, we
may assume the $p$-character $\chi$ satisfying $\chi(\ev \n^+)=0$ by
replacing it by a suitable $\ev G$-conjugate if necessary. A
$p$-character $\chi \in \ev \g^*$ is {\em semisimple} if it is $\ev
G$-conjugate to some $\xi \in \ev \g^*$ with $\xi(\ev \n^+) =
\xi(\ev \n^-)=0$, and $\chi$ is {\em nilpotent} if it is $\ev
G$-conjugate to some $\eta \in \ev \g^*$ with $\eta(\ev \n^+) =
\eta(\h)=0$.

Let $\chi \in \ev \g^*$ be such that $\chi(\ev \n^+)=0$. Then all
irreducible $\rea{\chi}{\mathfrak{b}}$-modules are one dimensional
of the form $K_{\la}$, where $\n^+$ acts trivially and $\h$ acts via
multiplication by $\la$ for some $\la \in \La_{\chi}$, and
\[
\La_{\chi}=\{\mu \in \h^*|\; \mu(h)^p-\mu(h)=\chi(h)^p \text{ for
all }h \in \h \}.
\]

Inducing the one-dimensional $\rea{\chi}{\mathfrak{b}}$-modules
$K_\la$ to $\rea{\chi}{\g}$, we get the baby Verma modules
\[
Z_{\chi}(\la)=\rea{\chi}{\g}\otimes_{\rea{\chi}{\mathfrak{b}}}K_{\la}.
\]

Given any $\rea{\chi}{\ev \g}$-module $M^0$, we view it as a
$\rea{\chi}{\g_+}$-module by letting $\g_1$ act on it trivially.
Inducing to $\rea{\chi}{\g}$, we get a $\rea{\chi}{\g}$-module
\[
\psi(M^0)= \rea{\chi}{\g} \otimes_{\rea{\chi}{\g_+}} M^0.
\]
In particular, when $Z^{0}_\chi(\la):= \rea{\chi}{\ev \g}
\otimes_{\rea{\chi}{\mathfrak{b}_{\bar{0}}}}K_{\la}$ is the baby
Verma module of $\rea{\chi}{\ev\g}$ with respect to
$\ev{\mathfrak{b}}$ with weight $\la \in \La_\chi$, we have
\begin{equation}\label{equ:verma}
Z_\chi(\la) \cong \psi(Z^{0}_\chi(\la)) =\rea{\chi}{\g}
\otimes_{\rea{\chi}{\g_+}} Z^{0}_\chi(\la).
\end{equation}

A weight $\la \in \La_{\chi}$ is called {\em typical} (cf.
\cite{Kac3}) if $(\la+\rho|\beta)\neq 0$ for all $\beta \in
\Delta^+_{\bar{1}}$; otherwise $\la$ is {\em atypical}.

Let $w \in W$ and $\la \in \La_\chi$. Then, by the $W$-invariance of
the bilinear form $(.|.)$, we have
\[
(w.\la+\rho|\beta)=(w(\la+\rho)|\beta)=(\la+\rho|w^{-1}\beta),
\quad\text{for } \beta \in \Delta^+_{\bar{1}}.
\]
Since $\Delta^+_{\bar{1}}$ is $W$-invariant, we deduce that $\la \in
\La_\chi$ is typical if and only if $w.\la$ is typical for any $w
\in W$.

\section{An equivalence of categories}\label{sec:ME}

\subsection{} Set
\[
T_{\pm}= \prod_{\beta \in \Delta_{\bar{1}}^{\pm}} X_{\beta},
\]
where the product is taken in any arbitrarily fixed order.

Let $U_\Z$ be the $\Z$-lattice of the universal enveloping algebra
$U_{\mathbb{C}}$ of the complex counterpart of $\g$ spanned by
elements of the form (\ref{equ:basis}) with $a_\alpha, b_i \in
\Z_{\geq 0}$ and $c_\beta =0,1$. Then we have
\[
U(\g)\cong U_\Z \otimes_\Z K.
\]
Denote by $\h_\Z$, $\n_\Z$, etc. the counterparts of $\h$, $\n$,
etc. over the integers. By abuse of notation, we use $h_i$,
$X_{\alpha}$ with $i=1,\ldots,r$ and $\alpha \in \Delta$ for the
integral bases for all three algebras $U_\Z$, $U_{\mathbb{C}}$, and
$U(\g)$.

\begin{lemma}\label{lem:top-polynomial}
We have
\[
T_+ T_- - P_1 \in  U_\Z \n^+_\Z,
\]
for a unique polynomial $P_1$ in variables $\{h_i\}_{i=1}^r$.
\end{lemma}
\begin{proof}
The universal enveloping algebra $U_\Z$ has a PBW basis which is
compatible with the product decomposition
$$U_\Z=U_\Z(\n_\Z^-)U_\Z(\h_\Z)U_\Z(\n_\Z^+),$$
where $U_\Z(\h_\Z)$ and $U_\Z(\n_\Z^{\pm})$ are the $\Z$-subalgebras
of $U_\Z$ spanned by products $\prod h_i^{b_i}$ with $b_i \in
\Z_{\geq 0}$ and $\prod_{\alpha \in \pm \Delta^{+}_{\bar 0}}
X_\alpha^{a_\alpha} \cdot \prod_{\beta \in  \pm \Delta^{+}_{\bar 1}}
X_\beta^{c_\beta}$ with $a_\alpha \in \Z_{\geq 0}$, $c_\beta =0,1$
respectively.

There is a natural grading on $U_\Z$ by the root lattice of $\g_\Z$
given by declaring $\text{deg}\, X_\alpha = \alpha$ for $\alpha \in
\Delta$ and $\text{deg}\,h_i=0$ for $i =1,\ldots, r$. Then the lemma
follows from writing the degree $0$ vector $T_+T_-$ in terms of the
above PBW basis.
\end{proof}

In \cite[Proposition~2.9]{Kac3}, Kac computed the polynomial $P_1$
for $U_{\mathbb C}$ over the complex numbers. By
Lemma~\ref{lem:top-polynomial}, we conclude that Kac's formula
\begin{equation}\label{equ:top-polynomial}
P_1(\la)= \pm \prod_{\beta \in \Delta^+_{\bar{1}}} (\la
+\rho|\,\beta), \quad \la \in \h^*,
\end{equation}
remains valid in our setting (of characteristic $p$).

Let $\chi \in \ev \g^*$ be such that $\chi(\ev \n^+)=0$. Then for
any simple $\rea{\chi}{\ev\g}$-module $L^0$, we have $(L^0)^{\ev
\n^+} \neq 0$. It may turn out that $(L^0)^{\ev \n^+}$ is not
one-dimensional, say we have $0 \neq v, u \in (L^0)^{\ev \n^+}$ with
$h.v=\la(h)v$ and $h.u=\mu(h)u$ for all $h \in \h$ and some $\la,
\mu \in \La_\chi$. Then $L^0$ is a quotient of both $Z^0_\chi(\la)$
and $Z^0_\chi(\mu)$. It follows that $U(\g)^{\ev G}$ acts on the two
baby Verma modules via the same central character. It is known (cf.
\cite[Corollary~9.4]{Jan1}) that $\la \in W.\mu$. Hence the weights
in $(L^0)^{\ev \n^+}$ are either all typical or are all atypical. A
simple $\rea{\chi}{\ev \g}$-module $L^0$ is called {\em typical} if
any of weights in $(L^0)^{\ev \n^+}$ is typical; otherwise, $L^0$ is
called {\em atypical}.

Now let $L^0$ be a simple $\rea{\chi}{\ev \g}$-module, and let $v$
be a nonzero vector in $(L^0)^{\ev \n^+}$ with $h.v=\la(h)v$ for all
$h \in \h$ and some $\la \in \La_\chi$. By
Lemma~\ref{lem:top-polynomial}, $T_+T_-$ acts on $v=1 \otimes v \in
\psi(L^0)$ by multiplication by $P_1(\la)$. Then by
(\ref{equ:top-polynomial}), we have $L^0$ is typical if and only if
$P_1(\la) \neq 0$.

\begin{lemma}\label{lem:commute}
The vector spaces $KT_+$, $KT_-$, and $KT_+T_-$ are one-dimensional
representations of $\rea{\chi}{\ev\g}$ under the adjoint action.
Among them, $KT_+T_-$ is the trivial one-dimensional representation.
\end{lemma}
\begin{proof}
The first assertion follows from $[\ev\g, \g_{\pm 1}] \subset
\g_{\pm 1}$ and $[\ev \g, \od \g] \subset \od \g$, and the fact that
$T_\pm$ and $T_+T_-$ lie in the top wedges of $\bigwedge \g_{\pm 1}$
and $\bigwedge \od \g$ respectively, where $\bigwedge \g_1$, etc. is
the exterior algebra of $\g_1$, etc..

To show that $KT_+T_-$ is a trivial $\rea{\chi}{\ev\g}$-module, note
first that elements in the semisimple part $[\ev \g,\ev\g]$ of $\ev
\g$ commute with $T_+T_-$. Elements in the subalgebra $\h$ also
commute with $T_+T_-$ since $T_+T_-$ has weight zero. Then any
element in $\ev \g$ commutes with $T_+T_-$. The second assertion is
proved.
\end{proof}

\begin{lemma}\label{lem:Kac module-low vector}
Let $L^0$ be an irreducible $\rea{\chi}{\ev \g}$-module. Then any
nonzero (graded or non-graded) $\rea{\chi}{\g}$-submodule $S$ of
$\psi(L^0)$ contains $T_- L^0$.
\end{lemma}
\begin{proof}
Note that, as vector spaces, $\psi(L^0)\cong
\rea{\chi}{\g_{-1}}\otimes L^0$. Starting from any vector $u \in S$,
we can then apply a suitable vector in $\rea{\chi}{\g_{-1}}$ to get
a vector of the form $T_-v$ for some $0 \neq v \in L^0$. The lemma
then follows from Lemma~\ref{lem:commute}, since
$\rea{\chi}{\ev\g}T_-v=T_-\rea{\chi}{\ev\g}v=T_-L^0$.
\end{proof}

\begin{corollary}\label{cor:typeM}
Let $L^0$ be a simple $\rea{\chi}{\ev\g}$-module. If the induced
module $\psi(L^0)$ is irreducible then it is of type $M$.
\end{corollary}
\begin{proof}
Follows immediately from Lemma~\ref{lem:Kac module-low vector}.
\end{proof}

\begin{proposition}\label{prop:Kac module}
Let $L^0$ be an irreducible $\rea{\chi}{\ev \g}$-module. Then the
induced $\rea{\chi}{\g}$-module $\psi(L^0)$ is irreducible if and
only if $L^0$ is typical.
\end{proposition}
\begin{proof}
Choose a nonzero weight vector $v \in (L^0)^{\ev\n^+}$ with $h.v =
\la(h)v$ for all $h \in \h$ and some $\la \in \La_\chi$. We claim
that:

{\bf Claim:} $\psi(L^0)$ is irreducible if and only if $v \in
T_+T_-L^0$.

According to Lemma~\ref{lem:Kac module-low vector}, it is easy to
see that the $\rea{\chi}{\g}$-module $\psi(L^0)$ is irreducible if
and only if $v \in \rea{\chi}{\g}T_-L^0$.

Now any element $u \in \rea{\chi}{\g}$ can be written as $u =u_1
\cdot u_0 \cdot u_{-1}$ with $u_{\pm 1} \in \rea{\chi}{\g_{\pm 1}}$
and $u_0 \in \rea{\chi}{\ev\g}$. Since $\rea{\chi}{\g_{-1}}\g_{-1}$
kills $T_-L^0$ and $\rea{\chi}{\ev\g}$ leaves $T_-L^0$ stable
(Lemma~\ref{lem:commute}), one sees that
\[
\rea{\chi}{\g}T_-L^0 = \rea{\chi}{\g_1}T_-L^0.
\]
Thus the $\rea{\chi}{\g}$-module $\psi(L^0)$ is irreducible if and
only if $v \in \rea{\chi}{\g_1}T_-L^0$.

Since $\rea{\chi}{\g_1}$ is a direct sum of $KT_+$ and the subspace
spanned by products of positive odd root vectors of length shorter
than the monomial $T_+$, we observe further that
\[
\rea{\chi}{\g_1}T_-L^0 \cap (1 \otimes L^0) = T_+T_-L^0 \cap (1
\otimes L^0).
\]
Thus $v\in \rea{\chi}{\g_1}T_-L^0$ if and only $v \in T_+T_-L^0$.
The Claim is proved.

Now if $L^0$ is typical, then $T_+T_-.v=P_1(\la)v \neq 0$. Thus $v
\in T_+T_-L^0$ and hence $\psi(L^0)$ is irreducible by the Claim.

Conversely, if $L^0$ is atypical then $T_+T_-.v=P_1(\la)v=0$.
Consider the subspace $T_+T_-L^0 \subset \psi(L^0)$. By
Lemma~\ref{lem:commute}, $T_+T_-L^0$ is a $\rea{\chi}{\ev
\g}$-submodule of $\psi(L^0)$. View $L^0 = 1 \otimes L^0$ also as a
$\rea{\chi}{\ev \g}$-submodule of $\psi(L^0)$. Then multiplying
elements of $L^0$ by $T_+T_-$ will give us a surjective map
\[
m: L^0 \ra T_+T_-L^0.
\]
The map $m$ is actually a $\rea{\chi}{\ev\g}$-module homomorphism
since $\rea{\chi}{\ev \g}$ commutes with $T_+T_-$
(Lemma~\ref{lem:commute}). It has a nontrivial kernel containing
$v$. Since $L^0$ is a simple $\rea{\chi}{\ev\g}$-module, $T_+T_-L^0$
has to be zero. By using the above Claim again, we conclude that
$\psi(L^0)$ is reducible.
\end{proof}


\subsection{}
The following proposition is an analogue of
\cite[Proposition~5.2.5(a)]{Kac1} in characteristic $p$.

\begin{proposition}\label{prop:g_1-functor}
Let $\chi \in \ev \g^*$ be an arbitrary $p$-character. Let $M$ be a
$\rea{\chi}{\g}$-module. Set $M^{\g_1}=\{m \in M|\; \g_1. m=0\}$. If
$M$ is irreducible, then $M^{\g_1}$ is irreducible as a
$\rea{\chi}{\ev \g}$-module. In particular, if $L^0$ is a typical
$\rea{\chi}{\ev\g}$-module, then
\[
L^0 \cong \psi(L^0)^{\g_1}.
\]
\end{proposition}
\begin{proof}
The proof is the same as that of \cite[Proposition~5.2.5(a)]{Kac1},
and will be omitted here.
\end{proof}

We call an irreducible $\rea{\chi}{\g}$-module $L$ {\em typical}
(resp. {\em atypical}) if $L^{\g_1}$ is typical (resp. atypical).
This is equivalent to saying that all weights in $L^{\n^+}$ are
typical (resp. atypical).


\begin{corollary}\label{cor:ME-simple-typical}
Let $\chi \in \ev\g^*$ be such that $\chi(\ev \n^+)=0$. Let $L$ be a
typical irreducible $\rea{\chi}{\g}$-module. Then $L$ is
$\rea{\chi}{\g_1}$-free.
\end{corollary}
\begin{proof}
By Proposition~\ref{prop:Kac module}, the natural map
$\psi(L^{\g_1}) \ra L$ is an isomorphism. Then $L$ is
$\rea{\chi}{\g_{-1}}$-free. Interchanging the roles of $\g_1$ and
$\g_{-1}$, it follows that $L$ is a free $\rea{\chi}{\g_1}$-module.
\end{proof}

\subsection{}
For $\rea{\chi}{\g}$-modules $M$ and $N$, we write $\text{Ext}(M,N)$
for the (degree 1) extensions between $N$ and $M$ in the category
$\rea{\chi}{\g}$-mod.

Here it is also worth noting that, since $\g_1$ is abelian,
$\rea{\chi}{\g_1}=\bigwedge\g_1$.

\begin{lemma}\label{lem:H1}
Let $L$ be a typical irreducible $\rea{\chi}{\g}$-module. Then
$H^1(\g_1, L)=0$.
\end{lemma}
\begin{proof}
Thanks to Corollary~\ref{cor:ME-simple-typical}, we only need to
prove that
\[
H^1(\g_1, \bigwedge\g_1)=0,
\]
which is a consequence of the following fact in linear algebra.
\end{proof}

\begin{lemma}\label{lem:exterior}
Let $V$ be a finite-dimensional vector space, and let $f : V \ra
\bigwedge V$ be a linear map satisfying $x \wedge f(x)=0$ for any $x
\in V$. Then $f(x)= x \wedge u$ for some $u \in \bigwedge V$ and all
$x \in V$.
\end{lemma}

For a $\rea{\chi}{\g}$-module $N$, $N^*$ denotes the dual
$\rea{-\chi}{\g}$-module.

\begin{lemma}\label{lem:dual}
Let $L$ be a typical irreducible $\rea{\chi}{\g}$-module. Then $L^*$
is a typical irreducible $\rea{-\chi}{\g}$-module.
\end{lemma}
\begin{proof}
First note by Proposition~\ref{prop:Kac module} that $L \cong
\psi(L^{\g_1})$. In particular, we have
\[
\dim L^* = \dim L = \dim \rea{\chi}{\g_{-1}} \cdot \dim L^{\g_1}.
\]

Next, by Corollary~\ref{cor:ME-simple-typical}, $L$ is a free
$\bigwedge\g_1$-module. It follows from $\dim (\bigwedge\g_1)^{\g_1}
= \dim ((\bigwedge\g_1)^*)^{\g_1}$ that
\[
\dim L^{\g_1} = \dim (L^*)^{\g_1},
\]
and so
\[
\dim L^*= \dim \rea{\chi}{\g_{-1}} \cdot \dim (L^*)^{\g_1}.
\]

Now, by dimension consideration, the natural surjective
$\rea{-\chi}{\g}$-module homomorphism
\[
\psi((L^*)^{\g_1}) \ra L^*
\]
has to be an isomorphism. This can happen if and only if
$(L^*)^{\g_{1}}$ is typical (Proposition~\ref{prop:Kac module}).
Hence $L^*$ is typical.
\end{proof}

\begin{lemma}\label{lem:extension}
We have $\text{Ext}(L, L')=0=\text{Ext}(L', L)$, for a typical
simple $\rea{\chi}{\g}$-module $L$ and an atypical simple
$\rea{\chi}{\g}$-module $L'$.
\end{lemma}
\begin{proof}
We prove first that $\text{Ext}(L', L)=0$. Let
\begin{equation}\label{equ:extension-g}
0 \ra L \ra M \ra L' \ra 0
\end{equation}
be an extension of $L'$ by $L$ in $\rea{\chi}{\g}$-mod.

By Lemma~\ref{lem:H1}, $H^1(\g_1,L)=0$, and so the long exact
sequence of $\g_1$-cohomology associated to the short exact
sequence~(\ref{equ:extension-g}) gives rise to a short exact
sequence of $\rea{\chi}{\ev\g}$-modules
\begin{equation}\label{equ:extension-g0}
0 \ra L^{\g_1} \ra M^{\g_1} \ra (L')^{\g_1} \ra 0.
\end{equation}
Note $L^{\g_1}$ and $(L')^{\g_1}$ are irreducible
$\rea{\chi}{\ev\g}$-modules by Proposition~\ref{prop:g_1-functor}.
The sequence (\ref{equ:extension-g0}) splits by the linkage
principle of $\rea{\chi}{\ev \g}$ (see, for example,
\cite[Propositions~C.4 and C.5]{Jan2}) since typical and atypical
weights are not linked by $W$.

View the $\rea{\chi}{\ev \g}$-modules as $\rea{\chi}{\g_+}$-modules
with trivial $\g_1$-action and induce to get a split exact sequence
(noting that the induction functor is exact) in $\rea{\chi}{\g}$-mod
\begin{equation}\label{equ:extension-go-ind}
0 \ra \psi(L^{\g_1}) \ra \psi(M^{\g_1}) \ra \psi((L')^{\g_1}) \ra 0.
\end{equation}
The Frobenius reciprocity allows us to complete
(\ref{equ:extension-g}) and (\ref{equ:extension-go-ind}) into a
commutative diagram with exact rows.
\[
\CD
 0 @> >> \psi(L^{\g_1}) @> >>  \psi(M^{\g_1}) @> >>  \psi((L')^{\g_1}) @> >>0\\
 && @V  VV @V  VV @V  VV  \\
 0@> >> L @> >> M @> >> L' @> >>0
\endCD
\]
By Proposition~\ref{prop:Kac module}, the first vertical arrow is
actually an isomorphism. Now it follows from the splitting of the
top row that the bottom row also splits.

To prove $\text{Ext}_{\rea{\chi}{\g}}(L,L')=0$, we simply note that
\[
\text{Ext}_{\rea{\chi}{\g}}(L,L')=
\text{Ext}_{\rea{-\chi}{\g}}((L')^*, L^*).
\]
Then the identity follows from Lemma~\ref{lem:dual} and the first
part of the proof.
\end{proof}

It follows from induction on the length of composition series that
the $\rea{\chi}{\g}$-modules (resp. $\rea{\chi}{\ev \g}$) whose
composition factors are all typical form a full subcategory of
$\rea{\chi}{\g}$-mod (resp. $\rea{\chi}{\ev\g}$-mod), which we
denote by $\rea{\chi}{\g}$-$\text{mod}^{\text{ty}}$ (resp.
$\rea{\chi}{\ev \g}$-$\text{mod}^{\text{ty}}$).

Now we are at a position to state and prove the main result of this
paper.

\begin{theorem}\label{thm:ME}
The functors
\begin{equation*}
\qquad\phi=-^{\g_1}: \rea{\chi}{\g}\text{-}\text{mod}^{\text{ty}}
\rightarrow \rea{\chi}{\ev \g}\text{-}\text{mod}^{\text{ty}}
\end{equation*}
and
\begin{equation*}\psi=\rea{\chi}{\g} \otimes_{\rea{\chi}{\g_+}} -:
\rea{\chi}{\ev \g}\text{-}\text{mod}^{\text{ty}} \rightarrow
\rea{\chi}{\g}\text{-}\text{mod}^{\text{ty}}
\end{equation*}
are inverse equivalence of categories.
\end{theorem}
\begin{proof}
By Propositions~\ref{prop:Kac module} and~\ref{prop:g_1-functor},
the adjunction morphisms $\text{id}_{\rea{\chi}{\ev
\g}\text{-}\text{mod}^{\text{ty}}} \ra \phi \circ \psi$ and $\psi
\circ \phi\ra
\text{id}_{\rea{\chi}{\g}\text{-}\text{mod}^{\text{ty}}}$ are
isomorphisms on irreducible objects. Since $\rea{\chi}{\g}$ is
$\rea{\chi}{\g_+}$-free, the functor $\psi$ is exact. It follows
from Corollary~\ref{cor:ME-simple-typical} that every module in
$\rea{\chi}{ \g}\text{-}\text{mod}^{\text{ty}}$ (being
finite-dimensional) is free over $\rea{\chi}{\g_1}$; thus $\phi$ is
also exact. To establish that the adjunction morphisms are
isomorphisms for general objects $M^0$ in
$\rea{\chi}{\ev\g}\text{-}\text{mod}^{\text{ty}}$ and $N$ in
$\rea{\chi}{\g}\text{-}\text{mod}^{\text{ty}}$, it remains to argue
inductively on the length of composition series for $M^0$ and $N$.
\end{proof}

\section{Applications of the equivalence}\label{sec:app}
In this section, we deduce various consequences of
Theorem~\ref{thm:ME}.
\subsection{} Let $\chi \in \ev \g^*$ be a semisimple $p$-character with
$\chi(\ev\n^+)=0=\chi(\ev\n^-)$. The even subalgebra $\ev \g$ of
$\g$ is a Lie algebra of a reductive group. It is shown by Rudakov
(\cite[Theorem~3]{Rud}) that a baby Verma module $Z^{0}_\chi(\la)$
for $\rea{\chi}{\ev \g}$ with $\la \in \La_\chi$ is irreducible if
and only if
\begin{equation}\label{equ:verma-g0}
(\la+\rho_0|\,\alpha) \notin \mathbb{F}_p \setminus\{0\}, \quad
\text{for all } \alpha \in \Delta^+_{\bar{0}},
\end{equation}
where $\mathbb{F}_p$ denotes the finite field of $p$ elements. The
original assumption on $\g$ is somewhat restrictive in \cite{Rud}.
This criterion was later improved by Friedlander-Parshall
(\cite[Theorem~3.5]{FP}). It holds as long as the $p$ is good for
$\ev\g$ (i.e. all $p$ for $\ev \g$ if $\g=\gl(m|n)$ or
$\mathfrak{sl}(m|n)$, and $p>2$ for $\g=\osp(2|2n)$). Note also that
Rudakov actually proved this criterion in the case of semisimple Lie
algebras, but it is easy to deduce from there the case when the Lie
algebra $\ev \g$ is reductive.

For $\la \in \La_\chi$, put
\[
P_0(\la) = \prod_{\alpha \in \Delta^+_{\bar{0}}}
((\la+\rho_0|\,\alpha)^{p-1} -1)=\prod_{\alpha \in
\Delta^+_{\bar{0}}} ((\la+\rho|\,\alpha)^{p-1} -1),
\]
then a baby Verma module $Z^{0}_\chi(\la)$ is irreducible if and
only if $P_0(\la) \neq 0$. Further put for $\la \in \La_\chi$
\[
P(\la)=P_0(\la)\cdot P_1(\la).
\]
.

\begin{theorem}\label{thm:irred-verma}
Let $\g$ be a basic classical Lie superalgebra of Type I, and let
$\chi \in \ev\g^*$ be semisimple with
$\chi(\ev\n^+)=0=\chi(\ev\n^-)$. Then for $\la \in \La_\chi$, the
baby Verma module $Z_\chi(\la)$ is irreducible if and only if
$P(\la) \neq 0$. Moreover, if $Z_\chi(\la)$ is irreducible, it is of
type $M$.
\end{theorem}
\begin{proof}
The first assertion follows from (\ref{equ:verma-g0}),
Proposition~\ref{prop:Kac module}, and the isomorphism
(\ref{equ:verma}). The second assertion follows from
(\ref{equ:verma}) and Corollary~\ref{cor:typeM}.
\end{proof}

For $\alpha \in \Delta$, denote by $H_\alpha =[X_\alpha,
X_{-\alpha}] \in \h$ the coroot of $\alpha$. The following lemma is
standard.

\begin{lemma}\label{lem:regular s.s.}
Let $\chi \in \ev \g^*$ be semisimple with $\chi(\ev
\n^+)=0=\chi(\ev \n^-)$. Then $\chi$ is regular semisimple (i.e.,
$\g_\chi:=\{y \in \g \vert\; \chi([y,\g])=0\}$ is a Cartan
subalgebra of $\g$) if and only if $\chi(H_\alpha)\neq 0$ for all
$\alpha \in \Delta^+$.
\end{lemma}

We now obtain a semisimplicity criterion for $\rea{\chi}{\g}$.

\begin{theorem}\label{thm:ss-criterion}
Let $\g$ be a basic classical Lie superalgebra of Type I, and let
$\chi \in \ev\g^*$ be semisimple with
$\chi(\ev\n^+)=0=\chi(\ev\n^-)$. Then $\rea{\chi}{\g}$ is semisimple
if and only if $\chi$ is regular semisimple.
\end{theorem}
\begin{proof}
Since $\chi$ satisfies $\chi(\ev\n^+)=0=\chi(\ev\n^-)$, each baby
Verma module $Z_\chi(\la)$ for $\la \in \La_\chi$ has a unique
irreducible quotient, and they form a complete and irredundant set
of irreducible $\rea{\chi}{\g}$-modules. Now by Wedderburn Theorem
and a dimension counting argument, $\rea{\chi}{\g}$ is semisimple if
and only if all the baby Verma modules $Z_\chi(\la)$ for $\la \in
\La_\chi$ are simple (of type $M$). By
Theorem~\ref{thm:irred-verma}, $Z_\chi(\la)$ being simple for all
$\la \in \La_\chi$ is equivalent to $P(\la)\neq 0$ for all $\la \in
\La_\chi$, and by the definition of $P$, this is equivalent to (i)
$(\la + \rho)(H_\alpha) \notin \mathbb{F}_p \setminus \{0\}$ for all
$\alpha \in \Delta_{\bar{0}}^+$ and (ii) $(\la+\rho)(H_\beta) \neq
0$ for all $\beta \in \Delta^+_{\bar{1}}$.

If $\chi$ is regular semisimple, then by Lemma~\ref{lem:regular
s.s.}, $\chi(H_\alpha) \neq 0$ for all $\alpha \in \Delta^+$. It
follows that for any $\la \in \La_\chi$, $\la(H_\alpha) \notin
\mathbb{F}_p$ for all $\alpha \in \Delta^+$ since
$\la(H_\alpha)^p-\la(H_\alpha)=\chi(H_\alpha)^p$. In this situation,
both (i) and (ii) are true since $\rho(H_\alpha) \in \mathbb{F}_p$
for any $\alpha \in \Delta^+$. Hence all $Z_\chi(\la)$ are simple
and $\rea{\chi}{\g}$ is semisimple.

Conversely, if $\chi$ is not regular semisimple, then
$\chi(H_\alpha)=0$ for some $\alpha \in \Delta^+$ by
Lemma~\ref{lem:regular s.s.}. Let us assume $\alpha \in
\Delta_{\bar{0}}^+$, since the other case can be argued in a similar
fashion. Then $\la(H_\alpha) \in \mathbb{F}_p$ for $\la \in
\La_\chi$. Since shifting the value of $\la(H_\alpha)$ by a number
in $\mathbb{F}_p$ will still result in an element in $\La_\chi$, we
may thus assume $(\la + \rho)(H_\alpha) =1$. Then $P(\la)=0$ and
$Z_\chi(\la)$ is reducible by Theorem~\ref{thm:irred-verma}. Hence
$\rea{\chi}{\g}$ is not semisimple.
\end{proof}

\begin{remark}
Note that the ``if'' part of the theorem is a consequence (cf.
\cite[Corollary~5.7]{WZ1}) of the Super Kac-Weisfeiler Property
(\cite[Theorem~5.6]{WZ1}).
\end{remark}
\begin{remark}
The simplicity criterion (Theorem~\ref{thm:irred-verma}) for baby
Verma modules and the semisimplicity criterion
(Theorem~\ref{thm:ss-criterion}) for reduced enveloping algebras
with semisimple $p$-characters are valid for queer Lie superalgebras
(\cite[Theorems~3.4 \& 3.10]{WZ2}).

For all basic classical Lie superalgebras, these corresponding
criterions can also be established, using the technique of odd
reflections and generalized arguments of Rudakov \cite{Rud}.
\end{remark}


\subsection{}
\begin{proposition}
Let $\g$ be a basic classical Lie superalgebra of Type I; and for
$\g=\mathfrak{sl}(m|n)$ we further assume that $p \nmid (m-n)$. Let
$\chi\in \ev \g^*$ with $\chi(\mathfrak{b}_{\bar{0}})=0$ be regular
nilpotent. Then the baby Verma module $Z_\chi(\la)$ is simple
provided $\la$ is typical.
\end{proposition}
\begin{proof}
Since $\chi$ is regular nilpotent for $\ev \g$ the baby Verma module
$Z^{0}_\chi(\la)$ for $\ev \g$ is simple by Premet's theorem
(\cite[Corollary~3.11]{Pr}). Hence $Z_\chi(\la)$ is simple by
(\ref{equ:verma}) and Proposition~\ref{prop:Kac module}.
\end{proof}

\subsection{}
Assume $\chi \in \ev \g^*$ is nilpotent with
$\chi(\ev{\mathfrak{b}})=0$. Recall that $\la \in \La_0$ is both
typical and regular if and only if $(\la+ \rho)(H_\alpha) \neq 0$
for all $\alpha \in \Delta^+$. The central character
$\text{cen}_\la$ on $U(\ev \g)^{\ev G}$ will determine a block in
$\rea{\chi}{\ev \g}\text{-}\text{mod}^{\text{ty}}$ and hence a block
in $\rea{\chi}{\g}\text{-}\text{mod}^{\text{ty}}$ via the
equivalence in Theorem~\ref{thm:ME}. Such a block is called typical
and regular.

\begin{remark}\label{rem:typical-regualr}
We note that if $\la \in \La_0$ is both typical and regular then the
characteristic $p$ of the base field $K$ is necessarily greater than
$h(\ev \g)$, the Coxeter number of $\ev \g$, which is $\max(m,n)$
for $\g=\gl(m|n)$ or $\mathfrak{sl}(m|n)$ and is $2n$ for
$\g=\osp(2|2n)$.
\end{remark}

Let $\ev B$ be the Borel subgroup of $\ev G$ with $\text{Lie}\, {\ev
B}=\ev{\mathfrak{b}}$, and let $\mathcal{B}^{0}$ be the flag
manifold of $\ev G$. Denote the Springer fiber of $\chi \in \ev
\g^*$ in $\mathcal{B}^0$ by $\mathcal{B}^0_\chi$.

\begin{theorem}\label{thm:BMR}
Let $\chi \in \ev \g^*$ be nilpotent with
$\chi(\ev{\mathfrak{b}})=0$. Then the number of simple modules in
any typical and regular block in
$\rea{\chi}{\g}\text{-}\text{mod}^{\text{ty}}$ is equal to the rank
of the Grothendieck group of the category of coherent sheaves on the
Springer fiber $\mathcal{B}^0_\chi$. This rank is also equal to the
dimension of the $l$-adic cohomology of $\mathcal{B}^0_\chi$.
\end{theorem}
\begin{proof}
By Theorem~\ref{thm:ME}, the theorem is equivalent to the
corresponding statement for $\rea{\chi}{\ev \g}$, which was
conjectured by Lusztig and established in
\cite[Corollary~5.4.3~and~Theorem~7.1.1]{BMR}.
\end{proof}

\subsection{} Assume now $\chi=0$. Let $T$ be the maximal torus of
$\ev G$ with $\text{Lie}\, T=\h$, and let $X$ be the character group
of $T$, with standard bilinear form $(,)$. The typicality of a
weight can also be extended to $X$ in an obvious way. The enveloping
algebra $U(\g)$ (resp. $U(\ev\g)$) is in a natural way $X$-graded
such that each $h \in \h$ has degree $0$ and each $X_\alpha$ has
degree $\alpha$ for $\alpha \in \Delta$ (resp. $\ev \Delta$). This
grading induces an $X$-grading on $\rea{0}{\g}$ (resp. $\rea{0}{\ev
\g}$).

For any $\la \in X$, define the graded baby Verma module
$\hat{Z}(\la)$ (resp. $\hat{Z}^0(\la)$) for $\rea{0}{\g}$ (resp.
$\rea{0}{\ev\g}$) in the same fashion as in
Section~\ref{sec:elerepn}, whose grading is given by declaring that
each basis element $\prod_{\alpha \in
\Delta_{\bar{0}}^+}X_{-\alpha}^{a_\alpha} \cdot \prod_{\beta \in
\Delta_{\bar{1}}^+}X_{-\beta}^{c_\beta} \otimes 1_\la$  (resp.
$\prod_{\Delta_{\bar{0}}^+}X_{-\alpha}^{a_\alpha} \otimes 1_\la$)
has degree $\la -\sum_{\alpha \in \ev \Delta^+} a_\alpha \alpha
-\sum_{\beta \in \od \Delta^+} c_\beta \beta$ (resp. $\la
-\sum_{\alpha \in \ev \Delta^+} a_\alpha \alpha $). For each
$\hat{Z}(\la)$ (resp. $\hat{Z}^0(\la)$), the unique irreducible
quotient $\hat{L}(\la)$ (resp. $\hat{L}^0(\la)$) is thus $X$-graded.
For $\la \in X$, denote
\[ \hat{V}(\la) := \hat{Z}(\la -2(p-1)\rho_0),
\quad (\text{resp. }\hat{V}^0(\la) = \hat{Z}^0(\la -2(p-1)\rho_0)).
\]
Note that
\[
\hat{Q}(\la)= \rea{0}{\g} \otimes_{\rea{0}{\g_+}} \hat{Q}^0(\la),
\]
for $Q=Z, L$ and $V$.

Denote by $\mathcal{C}$ (resp. $\mathcal{C}^0$) the category all
$X$-graded $\rea{0}{\g}$- (resp. $\rea{0}{\ev \g}$-)modules $M$
satisfying the condition that $\h$ acts on the $\gamma$-degree space
$M_\gamma$ via the tangent map $d(\gamma)$ for all $\gamma \in X$.
Thus $\mathcal{C}^0$ is identified with the category of $(\ev G)_1
T$-modules.

Let $W_p$ denote the affine Weyl group generated by the finite Weyl
group $W$ and the translations $s_{\alpha, np}$ with $\alpha \in \ev
\Delta$, $n \in \Z$. Denote the dot action of $W_p$ on $X$ by
$w.\la$ for $w \in W_p$ and $\la\in X$. Let $\uparrow$ be the
partial ordering on $X$ determined by reflections in the affine Weyl
group $W_p$ (cf. \cite[Section~D.7]{Jan2}).

For $\omega$ and $\tau$ in $W_p.\la$, let $\hat{P}_{\tau,\omega}$ be
the polynomial defined in \cite[Section~5]{CPS} which is in turn a
normalization of the generic Kazhdan-Lusztig polynomial of Kato
\cite{Kat}. Let $l$ be the length function on $W_p .\la$ defined in
\cite[Section~3.12.3]{CPS}.

\begin{theorem}
Assume $p > h(\ev \g)$, and assume that $\mathcal{C}^0$ has a
Kazhdan-Lusztig theory (in the sense of \cite[Definition~2.1]{CPS}).
Let $\la \in X$ be typical and satisfy $0 \leq (\la+\rho,
\alpha^{\vee})\leq p$. Then for any $\omega \in W_p .\la$, we have
\[
\text{ch}\, \hat{L}(\omega) = \sum_{\tau \uparrow
\omega}(-1)^{l(\tau)-l(\omega)} \hat{P}_{\tau, \omega}(-1)
\text{ch}\, \hat{V}(\tau).
\]
\end{theorem}
\begin{proof}
By Theorem~\ref{thm:ME}, the statement of the theorem is equivalent
to the corresponding one for $\rea{0}{\ev \g}$, which is established
in \cite[Theorem~5.7]{CPS}.
\end{proof}

\begin{remark}
We note that \cite[Theorem~5.7]{CPS} together with its partial
converse \cite[Theorem~5.9]{CPS} is essentially equivalent to the
Lusztig's conjecture \cite[Section~3, Problem~IV]{Lus} under
suitable assumption on $p$ (cf. \cite[Theorem~5.5]{CPS}).
\end{remark}

\end{document}